\documentclass[12pt]{amsart}
\usepackage[colorlinks=true,pagebackref,hyperindex,citecolor=Green,linkcolor=Blue]{hyperref}
\usepackage[top=1in, bottom=1in, left=1.1in, right=1.1in]{geometry}
\usepackage{amsmath}
\usepackage{amsfonts}
\usepackage{amssymb}
\usepackage{amsthm}
\usepackage[all]{xy}
\usepackage{mathrsfs}
\usepackage{enumitem} 
\usepackage[T1]{fontenc}
\usepackage{color}

\makeatletter
\def\namedlabel#1#2{\begingroup
#2%
\def\@currentlabel{#2}%
\phantomsection\label{#1}\endgroup
}
\makeatother

\theoremstyle{theorem} 
\newtheorem{theorem}{Theorem}[section]

\newtheorem{corollary}[theorem]{Corollary}

\newtheorem{lemma}[theorem]{Lemma}
\newtheorem{proposition}[theorem]{Proposition}

\newtheorem{theoremx}{Theorem}
\theoremstyle{definition} 
\newtheorem{definition}[theorem]{Definition}

\newtheorem{remark}[theorem]{Remark}
\numberwithin{equation}{subsection}

\definecolor{blue-violet}{rgb}{0.54, 0.17, 0.89}
\definecolor{Blue}{rgb}{0.01, 0.28, 1.0}
\definecolor{gGreen}{rgb}{0.2, 0.8, 0.2}
\definecolor{Green}{rgb}{0.04, 0.85, 0.32}

\newcommand{\sslash}{\mathbin{/\mkern-6mu/}}

\makeatletter
\def\@tocline#1#2#3#4#5#6#7{\relax
  \ifnum #1>\c@tocdepth % then omit
  \else
    \par \addpenalty\@secpenalty\addvspace{#2}%
    \begingroup \hyphenpenalty\@M
    \@ifempty{#4}{%
      \@tempdima\csname r@tocindent\number#1\endcsname\relax
    }{%
      \@tempdima#4\relax
    }%
    \parindent\z@ \leftskip#3\relax \advance\leftskip\@tempdima\relax
    \rightskip\@pnumwidth plus4em \parfillskip-\@pnumwidth
    #5\leavevmode\hskip-\@tempdima
      \ifcase #1
       \or\or \hskip 1.9em \or \hskip 2em \else \hskip 3em \fi%
      #6\nobreak\relax
    \dotfill\hbox to\@pnumwidth{\@tocpagenum{#7}}\par
    \nobreak
    \endgroup
  \fi}
\makeatother

\newcommand{\NN}{\mathbb{N}}

\newcommand{\ZZ}{\mathbb{Z}}

\newcommand{\KK}{\mathbb{K}}

\newcommand{\dif}[1]{^{{\langle #1\rangle}}}

\newcommand{\CC}{\mathbb{C}}
\newcommand{\m}{\mathfrak{m}}

\newcommand{\cA}{\mathscr{A}}

\newcommand{\cP}{{\mathcal P}}

\newcommand{\cO}{{\mathcal O}}
\newcommand{\cC}{{\mathcal C}}
\newcommand{\cJ}{{\mathcal J}}

\newcommand{\Spec}{\operatorname{Spec}}
\newcommand{\Hom}{\operatorname{Hom}}

\newcommand{\Rank}{\operatorname{rank}}

\newcommand{\fr}{\operatorname{free.rank}}

\newcommand{\ch}{\operatorname{char}}
\newcommand{\Der}{\operatorname{Der}}

\newcommand{\p}{\mathfrak{p}}

\newcommand{\Nash}{\operatorname{Nash}}

\newcommand{\Pnr}{\mathcal{P}^n_R}
\newcommand{\Pn}{\mathcal{P}^n_X}

\newcommand{\Grr}{\operatorname{Grass}}
\newcommand{\Gr}{\Grr_{\binom{n+d}{d}}(\mathcal{P}^n_X)}
\newcommand{\Ox}{\mathcal O}
\newcommand{\Oxx}{{\mathcal O}_{X,x}}

\newcommand{\I}{{\mathcal I}_X}
\newcommand{\V}{\mathbf{V}}

\newcommand{\diff}{\operatorname{diff}}

\begin{document}

\title[Nash blowups in prime characteristic]{Nash blowups in prime characteristic}

\author[D. Duarte]{Daniel Duarte$^1$}
\address{CONACyT-Universidad Aut\'onoma de Zacatecas, Zac., M\'exico}
\email{aduarte@uaz.edu.mx}

\author[L. N\'u\~nez-Betancourt]{Luis N\'u\~nez-Betancourt${^2}$}
\address{Centro de Investigaci\'on en Matem\'aticas, Guanajuato, Gto., M\'exico}
\email{luisnub@cimat.mx}

\thanks{{$^1$}Partially supported by CONACyT grant 287622.}

\thanks{{$^2$}Partially supported by CONACyT grant 284598 and C\'atedras Marcos Moshinsky.}


\maketitle
\setcounter{tocdepth}{1}

\begin{abstract}
We initiate the study of Nash blowups  in prime characteristic. First, we  show that a normal variety is non-singular if and only if its Nash blowup is an isomorphism, extending a theorem by A. Nobile. We also study higher Nash blowups, as defined by T. Yasuda. Specifically, we give a characteristic-free proof of  a higher version of Nobile's Theorem for quotient varieties and hypersurfaces. We also prove a weaker version for $F$-pure varieties.
\end{abstract}

\tableofcontents

%%%%%%%%%%%%%%%%%%%%%%%%%%%%%%%%%%%%%%%%%%%%%%%%%%%%%%%%%%%%%
\section{Introduction}
%%%%%%%%%%%%%%%%%%%%%%%%%%%%%%%%%%%%%%%%%%%%%%%%%%%%%%%%%%%%%

The Nash blowup is a natural modification of an algebraic variety that replaces singular points by limits of tangent spaces at non-singular points.
The main open problem in this topic is  whether the iteration of the Nash blowup solves the singularities of the variety. This question is usually attributed to J. Nash \cite{Nob} but it also appears in the work of J. G. Semple \cite{Semple}.  If true, it would give a canonical way to resolve singularities.
This problem has been an object of intense study \cite{Nob,Reb,GS1,GS2,HiroNash,Spiv,GlezTeis,GrigMil,Ataetal,DuarSurf}. 

In order to be able to achieve a resolution of singularities using Nash blowups, it is needed that this process always modifies a singular variety.
One of the first results that appeared in the theory of Nash blowups is Nobile's Theorem \cite{Nob}. It states that, for equidimensional varieties over $\CC$, the Nash blowup is an isomorphism if and only if the variety is non-singular. In additon to being  of central interest for the theory of Nash blowups, Nobile's Theorem has other applications. For instance, it appears in the study of link theoretic characterization of smoothness \cite{deFYu}. 

Unfortunately, Nobile's Theorem fails  over fields of positive characteristic. There are examples of singular curves over fields of prime characteristic whose Nash blowup is an isomorphism \cite{Nob}. Since the main goal of this theory  is to resolve singularities, these examples discouraged a further study of Nash blowups  in prime characteristic. 
%{\crd In order to have a similar  line of research in prime characteristic,   T. Yasuda \cite{YasudaFblowup} introduced a variant of the Nash blowup now called $F$-blowup.} {\cb Canijo, esta linea en rojo no me gusta mucho: le quita atencion al hecho que queremos discutir, es decir, revivir el Nash en caracteristica positiva e incluso juega contra nosotros: para que revivir el Nash si ya tienen el F-blowup en caracteristica positiva? Yo sugiero quitarla o ponerla en otro lugar que no sea la introduccion.}
One of the main purposes of this paper is to provide evidence that the classical  Nash blowup in prime characteristic behaves as expected after adding mild hypotheses. In our first main result we provide a version of Nobile's Theorem in prime characteristic for normal varieties.

\begin{theoremx}[{see Theorem \ref{ThmNobileNormal}}]\label{nob normal}
Let $X$ be a normal irreducible variety. If $\Nash_1(X)\cong X$, then $X$ is a non-singular variety. 
\end{theoremx}

We stress that the hypothesis of $X$ being normal is frequently assumed for many results in characteristic zero. For instance, M. Spivakovsky \cite{Spiv} showed that a sequence of normalized Nash blowups eventually give a resolution of singularities for surfaces. 
Theorem \ref{nob normal} implies that the original question regarding the Nash blowup and resolution of singularities can be reconsidered in arbitrary characteristic by iterating the normalized Nash blowup.

More recently, T. Yasuda \cite{Yas1} introduced a higher-order version of the Nash blowup, denoted as $\Nash_n(X)$, replacing tangent spaces by infinitesimal neighborhoods of order $n$. The main goal for this generalization was to investigate whether $\Nash_n(X)$ would give a one-step resolution of singularities for $n\gg 0$. This question has been settled recently for varieties over $\CC$: it has an affirmative answer for curves \cite{Yas1}, but it is false in general \cite{Toh}. Higher versions of Nobile's Theorem have been proved for some families of varieties \cite{DuarteToric,DuarteHyp,ChaDuaGil}.

On the other hand, T. Yasuda also showed that there are singular curves over fields of positive characteristic whose Nash blowup of any order is an isomorphism. In other words, a higher version of Nobile's theorem over fields of prime characteristic fails in general.  In our following results,  we show that a higher version of Nobile's theorem holds over fields of positive characteristic for some families of varieties.

Furthering the ideas in the proof of  Theorem \ref{nob normal}, we obtain  a weaker version of Nobile's Theorem for $F$-pure varieties and the higher Nash blowup.

\begin{theoremx}[{see Theorem \ref{ThmStReg} and Corollary \ref{CorNobileFpure}}]\label{nob f pure}
Let $X$ be an $F$-pure irreducible variety. If $\Nash_n(X)\cong X$, then $X$ is a strongly $F$-regular variety. In particular, if $\Nash_1(X)\cong X$ for some $n\geq 1$, then $X$ is a non-singular variety. 
\end{theoremx}

It is worth mentioning that strongly $F$-regular varieties have mild singularities, for instance, they are normal and Cohen-Macaulay. 

We also study a higher-order version of Nobile's Theorem for quotient varieties. In addition, we extend to fields of prime characteristic the  higher version of  Nobile's Theorem for normal hypersurfaces over $\CC$. We point out that the proofs of the following theorems are characteristic-free.

\begin{theoremx}[{see Theorem \ref{ThmNobileQuot}}]\label{group} Let $G$ be a finite non-trivial group such that 
$|G|$ has a multiplicative inverse in $K$.
Suppose that $G$ acts linearly on a polynomial ring $R$ over $\KK$, and  that $G$ contains no elements that fix a hyperplane in the space of one-forms $[R]_1$. Let $X=\Spec(R^G)$. Then, $\Nash_n(X)\not\cong X$.
\end{theoremx}

\begin{theoremx}[{see Theorem \ref{ThmNobileHyp}}]\label{nob hypers}
Let $X$ be a normal hypersurface. If $\Nash_n(X)\cong X$, then $X$ is a non-singular variety.
\end{theoremx}

Theorems \ref{group} and \ref{nob hypers} suggest that the question regarding one-step resolution via higher Nash blowups can be reconsidered in arbitrary characteristic for hypersurfaces and quotient varieties. In a second paper, we further study this question and higher-order versions of Nobile's Theorem in arbitrary characteristic for toric varieties. There we use combinatorial techniques, which differ significantly from the ones appearing in the present paper.

We end this introduction with a few comments about the techniques used in this manuscript.
B. Teissier \cite{Hunt} pioneered the use of derivations to study the Nash blowup in characteristic zero.  We further this line of research by using rings of diferential operators and  modules of principal parts in any characteristic. These techniques played a key role to prove all the main result in this manuscript. In particular, we use new developments in this line focused on singularities \cite{BJNB} and homological methods \cite{dACDuarte,BarDuar}. Furthermore,  we combine this approach with the use of Frobenius map to detect regularity \cite{Kunz} and certain type of singularities \cite{DModFSplit} to prove Theorems \ref{nob normal} and \ref{nob f pure}.

$ $

\noindent\textbf{Convention: }Throughout this paper, $\KK$ denotes an algebraically closed field and all varieties are assumed to be irreducible. In particular, $X$ always denotes an irreducible variety over $\KK$.
 We denote as $\NN$ the set of non-negative integers and $\ZZ^+$ the set of positive integers. By a local $\KK$-algebra $(R,\m,\KK)$, we mean a local ring $R$ with maximal ideal $\m$ such that $\KK\subseteq R$ and the map $\KK\hookrightarrow R\twoheadrightarrow  R/\m$ is an isomorphism.

%%%%%%%%%%%%%%%%%%%%%%%%%%%%%%%%%%%%%%%%%%%%%%%%%%%%%%%%%%%%%
\section{Nash blowups and  Nobile's Theorem}
%%%%%%%%%%%%%%%%%%%%%%%%%%%%%%%%%%%%%%%%%%%%%%%%%%%%%%%%%%%%%

In this section we recall the definition of Nash blowups of algebraic varieties. Then we discuss a classical theorem of A. Nobile \cite{Nob} in the theory of Nash blowups that characterizes smoothness in terms of these blowups.

Let $X$ be an irreducible algebraic variety of dimension $d$ over an algebraically closed field $\KK$ of arbitrary characteristic. Let $\I$ be the sheaf of ideals defining the diagonal $\Delta\hookrightarrow X\times X$. Let $\Pn:=\Ox_{X\times X}/\I^{n+1}$ be the sheaf of principal parts of order $n$ of $X$. Denote as $\Gr$ the Grassmanian of locally free quotients of $\Pn$ of rank $\binom{n+d}{d}$ and let $G_n:\Gr\to X$ be the structural morphism. 

The Grassmanian satisfies the following universal property \cite{EGAI}. Let $h:Y\to X$ be a morphism. Then $h^*\Pn$ has a locally free quotient of rank ${\binom{n+d}{d}}$ if and only if there exists $h':Y\to\Gr$ such that the following diagram commutes:
$$\xymatrix{Y \ar@{->}[r]\ar@{->}[dr]_{h} & \Gr \ar@{->}[d]^{G_n}\\ & X}$$

Let $U\subseteq X$ be the set of non-singular points of $X$, and let  $i:U\hookrightarrow X$ the inclusion morphism. Since $i^*\Pn=\Pn|_U$ is locally free of rank ${\binom{n+d}{d}}$, we have that there exists a morphism $\sigma:U\to\Gr$  by the universal property of  Grassmanians. 

\begin{definition}[{\cite{Nob, OnZa,Yas1}}]
Let $\Nash_n(X)$ denote the closure of $\sigma(U)$ in $\Gr$ with its reduced scheme structure, and let $\pi_n:\Nash_n(X) \to X$ be the restriction of $G_n$. We call $(\Nash_n(X),\pi_n)$ the \textit{Nash blowup of order n of} $X$.
\end{definition}

\begin{remark}
T. Yasuda defines the Nash blowup of order $n$ of $X$ using a different parameter space: the Hilbert scheme of points. Both definitions are equivalent \cite[Proposition 1.8]{Yas1}.
\end{remark}

The following theorem is a classical result in the theory of the usual Nash blowup.

\begin{theorem}[{Nobile's Theorem \cite{Nob}}]\label{Nob thm}
If $\ch(\KK)=0$, then  $\Nash_1(X)\cong X$ if and only if $X$ is non-singular.
\end{theorem}

There are generalizations of this result for $n\geq1$ in some cases \cite{DuarteToric,DuarteHyp,ChaDuaGil}. On the other hand, it is well known that this result is not true if $\ch(\KK)>0$. The classical counterexample is given by the cusp. If $X=\V(x^3-y^2)$ and $\ch(\KK)=2$, then $\Nash_n(X)\cong X$ for all $n\geq1$ (for $n=1$ this was proved by A. Nobile \cite{Nob}, and for $n\geq1$ by T. Yasuda \cite{Yas1}).

We are interested in studying analogs of Theorem \ref{Nob thm} for $n\geq1$  in arbitrary characteristic. Because of the previous example, it is necessary to add extra conditions on the variety if $\ch(\KK)>0$. We prove that Nobile's Theorem, or weaker versions of it, hold for some families of varieties. 

B. Teissier  gave a different proof of Nobile's Theorem \cite{Hunt} using the module of differentials and derivations in characteristic zero. A key part of his proof is that  $\Nash_1(X)\cong X$ implies that the module of differentials have a free summand of maximal rank. 
We give an extension to this fact to the  module of principal parts following the same ideas.
We give a proof of this result to stress that it is characteristic-free.

\begin{lemma}[\cite{Hunt}]\label{key}
Let $X$ be a variety of dimension $d$. Assume that  $\Nash_n(X)\cong X$ is an isomorphism. Then, 
$$\mathcal{P}^n_{\Oxx|\KK}\cong\Oxx^{\binom{n+d}{d}}\oplus T_x$$
for each $x\in X$,
where $\mathcal{P}^n_{\Oxx|\KK}$ is the module of principal parts of $\Oxx$ and $T_x$ is its torsion module.
\end{lemma}

\begin{proof}
Since  $\Nash_n (X)\cong X$, we have the following commutative diagram:
$$\xymatrix{X\ar[dr]_{Id}\ar[r]^{\pi_n^{-1}}&\Nash_n(X)\ar@{^(->}[r]\ar[d]^{\pi_n}&\Gr\ar[dl]^{G_n}\\&X}$$
By the universal property of the Grassmanian, $Id^*\Pn=\Pn$ has a locally free quotient of rank $\binom{n+d}{d}$.
Then, there exists a surjective morphism
$$\xymatrix{\Pn\ar[r]&\mathcal{L}\ar[r]&0},$$
where $\mathcal{L}$ is a locally free $\Ox_X$-module of rank $\binom{n+d}{d}$. Therefore, for each $x\in X$ the previous exact sequence induces
\begin{align}\label{surj}
\xymatrix{\mathcal{P}^n_{\Oxx|\KK}\ar[r]&\Oxx^{\binom{n+d}{d}}\ar[r]&0}.
\end{align}
Since $\Oxx^{\binom{n+d}{d}}$ is free and $\Rank(\mathcal{P}^n_{\Oxx|\KK})=\binom{n+d}{d}$, we conclude that $\mathcal{P}^n_{\Oxx|\KK}\cong\Oxx^{\binom{n+d}{d}}\oplus T_x,$ where $T_x$ is the torsion submodule of $\mathcal{P}^n_{\Oxx|\KK}$.
\end{proof}

\begin{remark}
For $n=1$ and $\ch(\KK)=0$, the existence of the surjective morphism (\ref{surj}) is used by B. Teissier to prove that $\Oxx$ is a regular local ring, implying Nobile's Theorem. The proof strongly uses a result by O. Zariski regarding derivations which  allows to apply induction on the dimension of the ring. Unfortunately, it is not clear how to extend Zariski's result for higher-order differential operators.
\end{remark}

%%%%%%%%%%%%%%%%%%%%%%%%%%%%%%%%%%%%%%%%%%%%%%%%%%%%%%%%%%%%%
\section{Analogs of Nobile's Theorem for normal  and $F$-pure varieties}
%%%%%%%%%%%%%%%%%%%%%%%%%%%%%%%%%%%%%%%%%%%%%%%%%%%%%%%%%%%%%

We start by recalling definitions and properties regarding differential operators that 
are used to prove Theorem \ref{ThmNobileNormal}.

\begin{definition}[\cite{EGA}]
Let $R$ be a   $\KK$-algebra.
The \textit{$\KK$-linear differential operators of $R$ of order $n$},\index{$D^n_{R \vert \KK}$}\index{$D_{R \vert \KK}$} $D^{n}_{R|\KK}\subseteq \Hom_\KK(R,R)$, are defined inductively as follows:
\begin{itemize} 
\item[(i)]  $D^{0}_{R|\KK} =\Hom_\KK (R,R).$ 
 \item[(ii)]  $D^{n}_{R|\KK} = \{\delta\in \Hom_\KK (R,R)\;|\; \delta r
 - r \delta \in D^{n-1}_{R|\KK} \;\forall \; r \in R \}.$ 
\end{itemize} 
The ring of $\KK$-linear differential operators is defined by $D_{R|\KK}=\displaystyle\bigcup_{n\in\NN}D^{n}_{R|\KK}$.
\end{definition}

\begin{definition}
Let $(R,\m,\KK)$ be a  local $\KK$-algebra with $\KK$ as a coefficient field.
We define the $n$-th differential powers \cite{SurveySP} of $\m$ by
$$
\m\dif{n}=\{f\in R \, | \, \delta(f)\in \m \hbox{ for all } \delta\in D^{n-1}_{R|\KK}\}
$$
for $n\in\ZZ^+$.
The differential core of $R$ \cite{BJNB} is defined by
$\p_{\diff}(R)=\bigcap_{n\in\ZZ^+} \m\dif{n}$.
\end{definition}

\begin{proposition}[{\cite[Proposition 4.15.]{BJNB}}]\label{PropDifPowerFreeRank}
Let $(R, \m, \KK)$ be a local $\KK$-algebra  with $\KK$ as a coefficient field.
Then,
$$
\dim_\KK (R/\m\dif{n+1})=\fr (\cP^n_{R|\KK}).
$$
\end{proposition}

We now present a perfect pairing between differential  operators and differential powers. This was implicitly introduced in previous work regarding convergence of differential signature \cite[Section 8]{BJNB}.

\begin{lemma}\label{LemmaPairing}
Let $(R,\m,\KK)$ be a  local $\KK$-algebra with $\KK$ as a coefficient field, and 
$\cJ_{R|\KK}=\{\delta \in D_{R|\KK} \;|\; \delta(R)\subseteq \m\}$.
There exists a non-degenerate $\KK$-bilinear function  
$$(\;\,\; ,\;\; ): D^{n-1}_{R|\KK}/  \cJ_{R|\KK}\cap D^{n-1}_{R|\KK} \times  R/\m\dif{n}    \to R/\m$$
defined by $( \overline{\delta },\overline{r} ) \mapsto \overline{\delta(r)}$.
\end{lemma}

\begin{proof}
By  definition, $D^{n-1}_{R|\KK} \m\dif{n} \subseteq \m $ and $ \cJ_{R|\KK} R\subseteq \m$. Then, $(\;\,\; ,\;\; )$
 is a well defined function. Since $\KK$-linearity in each  entry is given by the definition of $(\;\,\; ,\;\; )$, we focus on non-degeneracy.
 
 Given any $\delta \in  D^{n-1}_{R|\KK} \setminus \cJ_{R|\KK}$, there exists $r\in R\setminus  \m\dif{n}$ such that
 $\delta (r)\not\in\m$. Then,  $( \overline{\delta },\overline{r} ) \neq 0$.
 
  Similarly, for any $r\in R\setminus  \m\dif{n}$, there exists  $\delta \in  D^{n-1}_{R|\KK} \setminus \cJ_{R|\KK}$  such that
 $\delta (r)\not\in\m$. Then,  $( \overline{\delta },\overline{r} ) \neq 0$.
\end{proof}

We now introduce concepts in prime characteristic that play a role in the proof of Theorem \ref{ThmNobileNormal}.

\begin{definition}
Let $(R,\m,\KK)$ be a  local $\KK$-algebra with $\KK$ as a coefficient field. Suppose that $\KK$ has prime characteristic $p$.
Suppose that $R$ is a domain.
\begin{itemize}
\item The ring of $p^e$-roots of $R$ is defined by 
$$
R^{1/p^e}=\{f^{1/p^{e}} \; |\; f\in R \}\subseteq \overline{\hbox{frac}(R)}.
$$
\item We say that $R$ is $F$-finite if $R^{1/p^e}$ is finitely generated as an $R$-module.
\item We say that $R$ is $F$-pure if the inclusion $R\hookrightarrow R^{1/p^e}$ splits.
\item We say that $R$ is strongly $F$-regular if for every $c\in R\setminus \{0\}$ there exists $e\in\ZZ^+$ such that the inclusion $Rc^{1/p^e}\hookrightarrow R^{1/p^e}$ splits.
\end{itemize}
We say that a variety $X$ satisfies one of these properties if it is satisfied for every local ring $\cO_{X,x}$ for every closed point $x\in X$.
\end{definition}

\begin{definition}
Let $(R,\m,\KK)$ be a  local $\KK$-algebra with $\KK$ as a coefficient field.
Suppose that $\KK$ has prime characteristic $p$, and that $R$ is a domain.
\begin{itemize}
\item  We say that an additive map $\phi:R\to R$ is $p^{-e}$-linear if $\phi(r^{p^e} f)=r\phi(f).$
\item The set of  all $p^{-e}$-linear maps is denoted by  $\cC^e_R$.
\item The set of Cartier operators is defined by $\cC_R=\bigcup_{e\in \NN} \cC^e_R$.
\end{itemize}
\end{definition}

\begin{remark}\label{RemCorresCartier}
There is a bijective correspondence between 
$$\Psi:\cC^e_R\to\Hom_R(R^{1/p^e},R)$$ 
 given by $\Psi(\phi)(r^{1/p^e})=\phi(r)$
for $\phi\in\cC^e_R$.
\end{remark}

The following characterization of differential operators in prime characteristic plays a crucial role to relate them with Cartier operators.

\begin{theorem}[{\cite{Yek}}]
Suppose that $\KK$ is a perfect field.
Let $(R,\m,\KK)$ be a  local $\KK$-domain with $\KK$ as a coefficient field.
Then,
$$
D_{R|\KK}=\bigcup_{e\in\NN}\Hom_{R^{p^e}}(R,R).
$$
\end{theorem}

\begin{remark}\label{RemCorresLevel}
There is a bijective correspondence between 
$$\Psi:\Hom_{R^{p^e}}(R,R)\to\Hom_R(R^{1/p^e},R^{1/p^e})$$ 
 given by $\Psi(\phi)(r^{1/p^e})=(\phi(r))^{1/p^e}$
for $\phi\in \Hom_{R^{p^e}}(R,R)$.
\end{remark}

Now we are ready to prove that the Nash blow-up  does properly  modify normal varieties.

\begin{theorem}\label{ThmNobileNormal}
Let $X$ be a normal variety over $\KK$ of dimension $d$.
Suppose that $\KK$ has prime characteristic $p$.
If $\Nash_1(X)\cong X$, then $X$ is a non-singular variety. 
\end{theorem}
\begin{proof}
Let $x$ be a point in $X$. Let $R=\cO_{X,x}$ and $\m$ be its maximal ideal.
By Lemma \ref{key}, we have that the module of principal parts $\cP^1_{R|\KK}$ has a free summand of rank $d+1$.
As a consequence, the module of K\"{a}hler differentials  $\Omega_{R|\KK}\cong R^d\oplus N$ for some torsion  module $N$. 
As a consequence,  $\dim_\KK \left( \m /\m^{\langle 2\rangle}\right)=d$. 
There exist elements $x_1,\ldots,x_d\in \m$ and derivations $\partial_1,\ldots,\partial_d$ such that $\partial_i (x_j)$ is a unit if and only if $i=j$ by Lemma \ref{LemmaPairing}.
Let $A=(a_{i,j})$ be the $d\times d$-matrix  whose $(i,j)$-entry is $\partial_i (x_j)$. We note that $A$ is an invertible matrix, as it is invertible modulo $\m$.  Let $C=(c_{i,j})$ be the inverse of $A$. Let $\delta_k=\sum^d_{i=1} c_{k,i} \partial_i$.
Then, $\delta_t( x_j)=\sum^d_{i=1} c_{t,i} \partial_i( x_j)=\sum^d_{i=1} c_{k,i} a_{i,j}$. Then, $\delta_t( x_j)=1$ if 
$t=j$ and zero otherwise.

Let $\cA=\{\alpha=(\alpha_1,\ldots,\alpha_d)\in\mathbb{N}^d \; | \; \alpha_i<p\;\; \forall i \}.$
Let $\frac{1}{\alpha !}\delta^\alpha=\frac{1}{\alpha_1 !\cdots \alpha_d !} \delta^{\alpha_1 }_1\cdots  \delta^{\alpha_d }_d$ for $\alpha\in\cA$. We point out that $\frac{1}{\alpha_i !}$ in $\KK$ is well defined, because $\alpha_i<p$ for every $i$.
Since $\delta_t$ is a derivation for every $t$, we have that
$\frac{1}{\alpha!}\delta (x^\beta)=1$ if $\alpha=\beta$ and $\frac{1}{\alpha!}\delta^\alpha (x^\beta)\in \m$ if $\alpha\neq \beta$ for every  $\alpha,\beta\in \cA$.

Let $\widetilde{A}=(\widetilde{a}_{\alpha,\beta})$ be the $p^d\times p^d$-matrix indexed by $\cA\times \cA$, whose $(\alpha,\beta)$-entry is $\frac{1}{\alpha!}\partial^\alpha (x^\beta)$. For this, we need to order $\cA,$ but the choice of order does not play a role in the rest of the proof. 
We note that $\widetilde{A}$ is an invertible matrix.  Let $\widetilde{C}=(\widetilde{c}_{\alpha,\beta})$ be the inverse of $\widetilde{A}$. 
Let $\phi_\gamma=\sum_{\alpha} \widetilde{c}_{\gamma,\alpha}  \frac{1}{\alpha !}\partial^\alpha$.
Then, $\phi_\gamma( x^\beta)=\sum_{\alpha} \widetilde{c}_{\gamma,\alpha} \frac{1}{\alpha!} \partial_\alpha( x^\beta)=\sum_\alpha \widetilde{c}_{\gamma,\alpha} \widetilde{a}_{\alpha,\beta}$. Then, $\phi_\gamma( x^\beta)=1$ if 
$\gamma=\beta$ and zero otherwise.

Since $\KK$ is algebraically closed, $\Der_{R|\KK}\subseteq \Hom_{R^{p}}(R,R)$.
Moreover,  $\frac{1}{\alpha!}\delta^\alpha\in  \Hom_{R^{p}}(R,R)$ for every $\alpha\in \cA$.
As a consequence, $\phi_\alpha \in  \Hom_{R^{p}}(R,R)$ for every $\alpha\in \cA$.
Let $\varphi_\alpha \in \Hom_{R}(R^{1/p},R^{1/p})$  defined by 
$\varphi_\alpha(f^{1/p})=\left( \phi_\alpha (f)\right)^{1/p}.$
Then, $\varphi_\alpha( x^{\beta/p})=1$ if 
$\alpha=\beta$ and zero otherwise.

We set $\psi :\oplus_{\alpha\in\cA} Re_\alpha\to R^{1/p}$ defined by $e_\alpha\mapsto x^{\alpha/p}$.
Let $Q\subseteq R$ be a prime ideal of $R$ of height $1$.
Let $\psi _Q$ be the map induced by $\psi$ by the localization at $Q$.
Since $R$ is normal, we have that $R_Q$ is a regular ring.
Then, $R^{1/p}_Q$ is a free $R_Q$-module.
Let $\sigma_Q: R^{1/p}_Q\to R_Q$ be a splitting of the inclusion $R_Q\hookrightarrow R^{1/p}_Q$. 
We consider the map $\rho_Q: R^{1/p}_Q\to \bigoplus_{\alpha\in \cA} R_Qe_{\alpha}$ defined by 
$$
\rho_Q \left( \frac{f^{1/p}}{s} \right)=\bigoplus_{\alpha\in \cA}  \sigma_Q \left( \frac{\varphi_\alpha (f^{1/p}) }{s}\right) e_\alpha.
$$
Then, we have that $\rho_Q$ is surjective because $\rho_Q(x^{\alpha/p})=e_\alpha$.
Since  $R^{1/p}_Q$ is  a free $R_Q$-module  of rank $p^{d}$, we conclude that  $\rho_Q$ is an isomorphism.
Furthermore, $\rho_Q\circ \psi_Q$ is the identity on $\oplus_{\alpha\in\cA} R_Q e_\alpha$. As a consequence, $\psi_Q$ is an isomorphism. Since $R$ is normal, both $\oplus_{\alpha\in\cA} Re_\alpha$ and $R^{1/p}$ are torsion-free and  $(S_2)$. Then, $\psi$ is an isomorphism \cite[Lemma 15.23.14]{sp}. Hence,
$R^{1/p}$ is a free $R$-module, and so, $R$ is regular by Kunz's Theorem \cite{Kunz}.
\end{proof}

We now focus on $F$-pure varieties, that is, varieties whose local ring $\cO_{X,x}$ is $F$-pure for every closed point $x\in X$. For this, we need to recall  two criterions. One for $D$-simplicity and another for  strong $F$-regularity.

\begin{proposition}[{\cite[Corollary 3.16]{BJNB}}]\label{PropCoreDsimple}
Let $(R,\m,\KK)$ be a  local $\KK$-algebra with $\KK$ as a coefficient field.
Then, $R$ is simple as a $D_{R| \KK}$-module if and only if its differential core is zero.
%$\p_\diff (R)=0$. 
\end{proposition}

\begin{theorem}[{\cite[Theorem 2.2]{DModFSplit}}]\label{ThmSmith}
Let $(R,\m,\KK)$ be a  local $\KK$-algebra with $\KK$ as a coefficient field.
Let $R$ be an $F$-pure $F$-finite ring.
Then, $R$ is $D_{R|\KK}$-simple if and only if $R$ is strongly $F$-regular.
\end{theorem}

We now present another of our main results. Even though we are not able to show that $\Nash_n(X)\cong X$
implies smoothness, this condition implies strong $F$-regularity. In particular, in this case $X$ is Cohen-Macaulay and normal.

\begin{theorem}\label{ThmStReg}
Let $X$ be an $F$-pure variety. If $\Nash_n(X)\cong X$, then $X$ is a strongly $F$-regular variety. 
\end{theorem}
\begin{proof}
Let $x$ be a closed point in $X$, and $d=\dim(X)$. Let $R=\Ox_{X,x}$ and $\m$ be its maximal ideal. 
By Lemma \ref{key}, we have that the module of principal parts $\cP^n_{R|\KK}$ has a free summand of rank $\binom{n+d }{d}$.
Then, by Lemma \ref{key}, $\dim_\KK(R/\m^{\langle n\rangle })=\binom{n+d }{ d}$. Let $\p$ denote the differential core of $R$. 

% Then, $\p\subseteq \m^{\langle n\rangle}$.

Let $\overline{R}=R/\p$ and $\overline{\m}=\m \overline{R}$.
Since $\p$ is a $D_{R|\KK}$-ideal, we have a natural map of filtered rings $D_{R|\KK}\to D_{\overline{R}| \KK}$. Then,
$\overline{\m}^{\langle n\rangle }\subseteq \m^{\langle n\rangle }\overline{R}$.
Since $\p=\bigcap_{t\in\NN} \m^{\langle t\rangle }$, we have that
$$
\binom{n+d }{ d}=\dim_\KK (R/\m^{\langle n\rangle })=\dim_\KK(R/\m^{\langle n\rangle}+\mathfrak{p})
\leq \dim_\KK (R/\overline{\m}^{\langle n\rangle })\leq \binom{n+c }{ c},
$$
where $c=\dim R/\p\leq d$  by Proposition \ref{PropDifPowerFreeRank}. Then, $c=d$.
We note that $\p$ contains every minimal prime. 
We conclude that $\p=0$; otherwise,  $\p$  contains a parameter and $c<d$.  Hence, $R$ is strongly $F$-regular 
by Proposition \ref{PropCoreDsimple} and Theorem \ref{ThmSmith}.
\end{proof}

We now present an analogous to Nobile's Theorem for $F$-pure rings. It is worth mentioning that $F$-pure rings might not be normal. 

\begin{corollary}\label{CorNobileFpure}
Let $X$ be an $F$-pure  variety. If $\Nash_1(X)\cong X$, then $X$ is a non-singular variety. 
\end{corollary}
\begin{proof}
By Theorem \ref{ThmStReg}. $X$ is a strongly $F$-regular variety. Then, $X$ is a normal variety.
Hence, $X$ is  nonsingular by Theorem \ref{ThmNobileNormal}.
\end{proof}

%%%%%%%%%%%%%%%%%%%%%%%%%%%%%%%%%%%%%%%%%%%%%%%%%%%%%%%%%%%%%
\section{Higher-order versions of Nobile's Theorem}
%%%%%%%%%%%%%%%%%%%%%%%%%%%%%%%%%%%%%%%%%%%%%%%%%%%%%%%%%%%%%

In this section we study a higher version of Nobile's Theorem for quotient varieties. We also revisit a known results for hypersurfaces \cite[Theorem 4.13]{DuarteHyp} concerning the analog of Nobile's Theorem for higher Nash blow-ups in prime characteristic.

\subsection{Quotient varieties}

Let $G$ be a  linearly reductive algebraic group acting algebraically on $\Spec(R)$, where $R$ is a polynomial rings over $\KK$. The algebraic quotient $X\sslash G$ is defined by identifying two points of X whenever their orbit closures have non-empty intersection. 
This is an  affine algebraic variety whose coordinate ring is $R^G$.
If all the orbits are closed, then $X\sslash G$ is the usual orbit space and it is called a quotient variety.
If $|G|$ has a multiplicative inverse  in $\KK$, this situation happens. In this subsection, we present a higher version of  Nobile's Theorem in this case.

\begin{theorem}\label{ThmNobileQuot}
Let $G$ be a finite non-trivial  group such that $|G|$ has a multiplicative inverse in $K$. Let $G$ act linearly on a polynomial ring $R=\KK[x_1,\ldots, x_d]$. 
Suppose that $G\setminus\{e\}$ contains no elements that fix a hyperplane in the space of one-forms $[R]_1$. 
Let $X=\Spec(R^G)$.
Then, $\Nash_n(X)\not\cong X$.
\end{theorem}
\begin{proof}
The ramification locus of a finite group action corresponds to the union of fixed spaces of elements of $G$. Consequently, the assumption that no element fixes a hyper-plane ensures that the extension is unramified in codimension one. The inclusion is
order-differentially extensible \cite[Proposition 6.4]{BJNB}.
Let $\m$ be the maximal homogeneous ideal of $R$, and $\eta=\m\cap R^G$.
We note that $\dim R^G_\m=\dim R^G=d$.
Under these conditions, we have that $\eta^{\langle n\rangle }=\m^n\cap R^G$ 
\cite[Proposition 6.14]{BJNB}, 
Then, 
\begin{equation}\label{EqG1}
\dim_\KK  (\eta^{\langle j-1\rangle }_\eta/\eta^{\langle j\rangle }_\eta) \leq  \dim_\KK ( \m^{j-1}_\m/\m^j_\m).
\end{equation}
By our assumptions on $G$, we have that $R\neq R^G$. Then, 
\begin{equation}\label{EqG1}
\dim_\KK (\eta^{\langle 1\rangle }_\eta/\eta^{\langle 2\rangle }_\eta) <\dim_\KK (\m_\m/\m^ 2_\m)=d.
\end{equation}
We conclude that 
$$
\dim_\KK (R/\eta^{\langle n\rangle }_\eta)=\sum^n_{j=1} \dim_\KK (\eta^{ \langle j-1\rangle}_\eta/\eta^{\langle j\rangle}_\eta)<\sum^n_{j=1} \dim_\KK (\m^{j-1}_\m/\m^j_\m)=\binom{n+d}{d}.
$$
Then, the free rank of  $\cP_{R^G_\eta |\KK}$ is strictly smaller than $\binom{n+d}{d}$ by Proposition \ref{PropDifPowerFreeRank}. As a consequence, $\Nash_n(X)\not\cong X$ by Lemma \ref{key}.
\end{proof}

\subsection{Hypersurfaces}

Now we study the case of hypersurfaces. We note that the proof we present is characteristic free. 

\begin{theorem}\label{ThmNobileHyp}
Let $X$ be a normal hypersurface. If $\Nash_n(X)\cong X$, then $X$ is a non-singular variety.
\end{theorem}
\begin{proof}
Let $x\in X$ and $R=\Oxx$. By Lemma \ref{key}, $\Pnr\cong R^{\binom{n+d}{d}}\oplus T$, where $T$ is the torsion submodule. On the other hand, $R$ normal implies that $\Pnr$ is torsion-free \cite[Theorem 4.3]{BarDuar}. Therefore $\Pnr\cong R^{\binom{n+d}{d}}$. Then $R$ is a regular ring \cite[Theorem 3.1]{BarDuar} (see also \cite[Theorem 10.2]{BJNB} for a more general statement). We conclude that $X$ is non-singular.
\end{proof}

\section*{Acknowledgments}

We thank Josep \`Alvarez Montaner, and Jack Jeffries for helpful comments.

\bibliographystyle{alpha}
\bibliography{References}

\end{document}